\documentclass[10pt]{amsart}

\usepackage{amsthm}
\usepackage{amsmath}
\usepackage{amssymb}

\newtheorem{thm}{Theorem}[section]
\newtheorem{theorem}[thm]{Theorem}

\newtheorem{proposition}[thm]{Proposition}

\theoremstyle{remark}
\newtheorem{remark}[thm]{Remark}

\newcommand{\RR}{\mathbb R}

\title[Hessenberg Unitary Matrices]
{A Simple algorithm for constructing all real Hessenberg unitary matrices}
\author[J. Tremain
 ]{Janet C. Tremain}
\address{Department of Mathematics, University
of Missouri, Columbia, MO 65211-4100}

\thanks{The author was supported by DTRA/NSF 1042701}

\email{j.tremain@mchsi.com}

\begin{document}

\maketitle

\begin{abstract}
Unitary matrices which are zero below the secondary diagonal (Hessenberg
unitary matrices) have many
uses in analysis.  Given a set of needed conditions on a unitary matrix, this algorithm
will give the sparsest unitary matrix.
We give an algorithm for constructing all real Hessenberg unitary matrices.
The $n\times n$ unitary matrices given by the algorithm have $n-1$ variables which can be chosen
to give additional properties needed for a particular application.
\end{abstract}

\section{Introduction}

Hessenberg matrices have may uses in both pure an applied mathematics.
It is a well used fact that every matrix is unitarily equivalent to such a matrix and
there are polynomial time algorithms for doing this. 
In numerical linear algebra, they are used to speed-up eigenvalue computations.
Also, real Hessenberg unitary matrices have
various uses in analysis.  They also have the advantage of being quite sparse.
Here we give an algorithm which constructs all such real unitary matrices.
The algorithm starts with $2\times 2$ matrices and for each new dimension,
we drop the top row of the previous dimension, introduce a new variable and
then construct two new rows with the new variable added and finally use
the rest of the rows of the previous case by adding the first column of zeroes.
The $n\times n$ unitary matrices have $n-1$ variables which can be chosen
to give additional properties needed for a particular application.

This algorithm was developed at the request of the people working in the 
{\it Frame Research Center} (www.framerc.org) for their work in frame theory
where they needed the sparsest unitary matrices.

\begin{remark}
Since we can multiply rows and columns of a unitary by $-1$ and still have a 
unitary, and if we permute rows or columns while maintaining the
Hessenberg property we still have a unitary,
we will construct all real triangular unitary matrices {\it up to} such
changes in sign and permutations of rows or columns.
\end{remark}

\section{The Construction}

We will give an algorithm which constructs all real Hessenberg unitary matrices.  So we
can better see how the algorithm works, we will start by giving some small examples.

\subsection{$2\times2$ Matrices}

This is a unique class and is given by:

\[\begin{bmatrix}
\sqrt{1-z_1}&\sqrt{z_1}\\
-\sqrt{z_1}&\sqrt{1-z_1}
\end{bmatrix}\]

Here we must have $0\le z_1 \le1$.  

\subsection{$3\times 3$ Matrices}

This class is given by:
\[\begin{bmatrix}
\sqrt{1-z_2}& \sqrt{(1-z_1)z_2}& \sqrt{z_1z_2}\\
-\sqrt{z_2}& \sqrt{(1-z_2)(1-z_1)}& \sqrt{(1-z_2)z_1}\\
0& -\sqrt{z_1}& \sqrt{1-z_1}
\end{bmatrix}\]
Here, $0\le z_1,z_2 \le1$.

\begin{proposition}
This class of matrices consists of unitary matrices.
\end{proposition}

\begin{proof}
The proof is done by cases:
\vskip12pt
\noindent {\bf Case 1}:  The rows have norm 1.
\vskip12pt
The square sum of the elements of row 1 is
\[ (1-z_2)+(1-z_1)z_2+z_1z_2 = 1.
\]
The square sum of the elements of row 2 is
\[ z_2+(1-z_2)(1-z_1) + (1-z_2)z_1 = z_2+(1-z_2) =1.
\]
It is clear that row 3 is norm 1.
\vskip12pt
Because of the symmetry of the matrix (I.e. switching $z_1$ and
$z_2$ in the calculation), it is clear that the row sums and column sums
are equal.
\vskip12pt
\noindent {\bf Case 2}:  The rows are orthogonal.
\vskip12pt
The inner product of rows 1 and 2 is
\[ -\sqrt{(1-z_2)z_2}+ (1-z_1)\sqrt{(1-z_2)z_2}+z_1\sqrt{(1-z_2)z_2}=
\]
\[ \sqrt{(1-z_2)z_2}[-1+(1-z_1)+z_1]=0.\]
The inner product of rows 1 and 3 is
\[ -\sqrt{(1-z_1)z_1z_2}+\sqrt{(1-z_1)z_1z_2}=0.
\]
The inner product of rows 2 and 3 is
\[ -\sqrt{(1-z_2)(1-z_1)z_1}+ \sqrt{(1-z_2)(1-z_1)z_1}=0.\]
\vskip12pt
By symmetry again, the columns are orthogonal.
\end{proof}

We give some examples of how to use this:

\noindent {\bf Case 1}:  If $z_1=z_2=0$ we get
\[\begin{bmatrix}
1&0&0\\
0&1&0\\
0&0&1
\end{bmatrix}\]

\noindent {\bf Case 2}:  If we let $z_1=z_2=1$ we get
\[ \begin{bmatrix}
0&0&1\\
-1&0&0\\
0&-1&0
\end{bmatrix}\]

\noindent {\bf Case 3}:  If we let $z_1=1$ and $z_2=0$ we get
\[\begin{bmatrix}
1&0&0\\
0&0&1\\
0&-1&0
\end{bmatrix}\]

\noindent {\bf Case 4}:  If we let $z_1 = \frac{1}{2}$ and $z_2=1$ we get
\[\begin{bmatrix}
0& \sqrt{\frac{1}{2}}&\sqrt{\frac{1}{2}}\\
-1&0&0\\
0&-\sqrt{\frac{1}{2}}&\sqrt{\frac{1}{2}}
\end{bmatrix}\]

\noindent {\bf Case 5}:  If $z_1=\frac{1}{2}$ and $z_2= \frac{2}{3}$ we get
\[\begin{bmatrix}
\sqrt{\frac{1}{3}}&\sqrt{\frac{1}{3}}&\sqrt{\frac{1}{3}}\\
-\sqrt{\frac{2}{3}}&\sqrt{\frac{1}{6}}&\sqrt{\frac{1}{6}}\\
0& - \sqrt{\frac{1}{2}}&\sqrt{\frac{1}{2}}
\end{bmatrix}\]

\begin{theorem}
This algorithm gives all $3\times 3$ real Hessenberg unitary matrices.
\end{theorem}

\begin{proof}
We will just outline this.  Given a unitary
\[\begin{bmatrix}
a_{11}&a_{12}&a_{13}\\
a_{21}&a_{22}&a_{23}\\
0&a_{32}&a_{33}
\end{bmatrix}\]
If one of $a_{32}$ or $a_{33}$ is zero, it is easily checked that the matrix is
a permutation of the rows or columns of the identity multiplied by $-1$ if necessary
and this matrix easily arises from our algorithm.  Otherwise,
we may assume $a_{32}$ is negative and let our $z_1= a_{32}^2$ to get the
last row correct given that it square sums to 1.  Now, the last row must be orthogonal
to the rows above it so $(a_{22},a_{23})= c(a_{12},a_{13})$ are the unique vectors
(up to length) in $\RR^2$ which are
orthogonal to row 3.  Since $z_1$ is given, we can find the unique $z_2$ which makes
the last two terms of our first and second rows equal to the given one.  The first column
is now uniquely determined and must be of our form.
\end{proof}

\subsection{$4\times 4$ Matrices}

We construct a unitary matrix by:

\[\begin{bmatrix}
\sqrt{1-z_3}& \sqrt{(1-z_2)z_3}&\sqrt{(1-z_1)z_2z_3}& \sqrt{z_1z_2z_3}\\
- \sqrt{z_3}& \sqrt{(1-z_3)(1-z_2)}& \sqrt{(1-z_3)(1-z_1)z_2}& \sqrt{(1-z_3)z_1z_2}\\
0& -\sqrt{z_2}& \sqrt{(1-z_2)(1-z_1)}& \sqrt{(1-z_2)z_1}\\
0&0& -\sqrt{z_1}&\sqrt{1-z_1}
\end{bmatrix}\]

We leave it to the reader to check that this constructs all $4\times 4$ real Hessenberg
unitary matrices.

\subsection{$5\times 5$ Matrices:  The Algorithm}

  Given our general $n\times n$ real Hessenberg unitary
matrix with $n-1$ variables in it, we introduce
 a new {\it multiplier} $\sqrt{z_{n+1}}$ on the second row of the $n\times n$ case and the
first entry in the second row of the $(n+1)\times (n+1)$ matrix is $-\sqrt{z_{n+1}}$.  We also need the
sum of the squares of the entries of the two new rows we are adding at the top to equal
the squares of the entries of the top row of the $n\times n$ matrix.  After that, we add the
remaining rows of the previous case with a column of zeroes in front.

So, to pass from the $4\times 4$ case to the $5\times 5$ case we create two new rows:

\[\begin{bmatrix}
\sqrt{1-z_4}&\sqrt{(1-z_3)z_4}& \sqrt{(1-z_2)z_3z_4}& \sqrt{(1-z_1)z_2z_3z_4}& \sqrt{z_1z_2z_3z_4}\\
-\sqrt{z_4}& \sqrt{(1-z_4)(1-z_3)}& \sqrt{(1-z_4)(1-z_2)z_3}& \sqrt{(1-z_4)(1-z_1)z_2z_3}&
\sqrt{(1-z_4)z_1z_2z_3}
\end{bmatrix}\] 

Combining this with the previous case we get:

\[\begin{bmatrix}
\sqrt{1-z_4}&\sqrt{(1-z_3)z_4}& \sqrt{(1-z_2)z_3z_4}& \sqrt{(1-z_1)z_2z_3z_4}& \sqrt{z_1z_2z_3z_4}\\
-\sqrt{z_4}& \sqrt{(1-z_4)(1-z_3)}& \sqrt{(1-z_4)(1-z_2)z_3}& \sqrt{(1-z_4)(1-z_1)z_2z_3}&
\sqrt{(1-z_4)z_1z_2z_3}\\
0&- \sqrt{z_3}& \sqrt{(1-z_3)(1-z_2)}& \sqrt{(1-z_3)(1-z_1)z_2}& \sqrt{(1-z_3)z_1z_2}\\
0&0& -\sqrt{z_2}& \sqrt{(1-z_2)(1-z_1)}& \sqrt{(1-z_2)z_1}\\
0&0&0& -\sqrt{z_1}&\sqrt{1-z_1}

\end{bmatrix}\] 

\begin{theorem}
The matrix given above is a unitary matrix.
\end{theorem}

\begin{proof}
We will outline the proof.  Since the last three rows come from the previous unitary matrix,
we just need to check that the first two rows are orthogonal to all the others, that they are
norm 1 and the column vectors have norm 1.

\vskip12pt
\noindent {\bf The two new rows are norm 1}:

The norm of row 1 is:
\[ 1-z_4+(1-z_3)z_4 +(1-z_2)z_3z_4 + (1-z_1)z_2z_3z_4 + z_1z_2z_3z_4 =\]
\[ 1-z_3z_4+ (1-z_2)z_3z_4 + z_2z_3z_4 =
1-z_3z_4 +z_3z_4 = 1.
\]
The norm of row 2 is:
\[ z_4+(1-z_4)(1-z_3) + (1-z_4)(1-z_2)z_3+ (1-z_4)(1-z_1)z_2z_3 + (1-z_4)z_1z_2z_3=\]
\[z_4 + (1-z_4)[(1-z_3)+(1-z_2)z_3 + (1-z_1)z_2z_3 + z_1z_2z_3]=\]
\[ z_4+(1-z_4)[1-z_2z_3 + z_2z_3]= 1.
\]
\vskip12pt
\noindent {\bf The column vectors are norm 1}:

Column 1 square sums to $(1-z_4)+z_4=1$.  Column 2 square sums to
\[ (1-z_3)z_4+(1-z_4)(1-z_3) + z_3= 1-z_3+z_3=1.
\]
Column 3 square sums to
\[ (1-z_2)z_3z_4 + (1-z_4)(1-z_2)z_3 + (1-z_3)(1-z_2) + z_2=
\]
\[ (1-z_2)[z_3z_4+(1-z_4)z_3 + (1-z_3)]+z_2= (1-z_2)[1] + z_2=1.
\]
Column 4 square sums to
\[ (1-z_1)z_2z_3z_4 + (1-z_4)(1-z_1)z_2z_3+ (1-z_3)(1-z_1)z_2+
(1-z_2)(1-z_1) + z_1=\]
\[ (1-z_1)[z_2z_3z_4 + (1-z_4)z_2z_3+(1-z_3)z_2+(1-z_2)]+z_1=
\]
\[ (1-z_1)[z_2z_3+ 1-z_2z_3]+z_1= 1.
\]
Column 5 square sums to
\[ z_1z_2z_3z_4+ (1-z_4)z_1z_2z_3+ (1-z_3)z_1z_2+(1-z_2)z_1+(1-z_1)=
\]
\[ z_1z_2z_3 + z_1z_2-z_1z_2z_3 + x_1-z_1z_2 + 1-z_1= 1.
\]

\vskip12pt
\noindent {\bf The rows are orthogonal}:
The inner product of rows 1 and 2 is:
\[ -\sqrt{(z_4(1-z_4)}+\sqrt{(1-z_3)^2z_4(1-z_4)}+\sqrt{(1-z_2)^2z_3^2z_4(1-z_4)}+\]
\[
\sqrt{(1-z_1)^2z_2^2z_3^2z_4(1-z_4)}+\sqrt{z_1^2z_2^2z_3^2z_4(1-z_4)}=
\]
\[ -\sqrt{z_4(1-z_4)}+ (1-z_3)\sqrt{z_4(1-z_4)}+ (1-z_2)z_3\sqrt{z_4(1-z_4)}+\]
\[ (1-z_1)z_2z_3\sqrt{z_4(1-z_4)}+z_1z_2z_3\sqrt{z_4(1-z_4)}=\]
\[ \sqrt{z_4(1-z_4)}[ -1+(1-z_3)+(1-z_2)z_3+(1-z_1)z_2z_3+z_1z_2z_3]=
\]
\[ \sqrt{z_4(1-z_4)}[-1+(1-z_3)+z_3-z_2z_3+ z_2z_3-z_1z_2z_3+z_1z_2z_3]=0.
\]
\end{proof}

The inner product of rows 1 and 3 and taking our squares as above at the same time is
\[ -\sqrt{z_3(1-z_3)z_4}+ (1-z_2)\sqrt{z_3(1-z_3)z_4}+ (1-z_1)z_2\sqrt{x_3(1-z_3)z_4}+
z_1z_2\sqrt{z_3(1-z_3)z_4}=
\]
\[ \sqrt{z_3(1-z_3)z_4 }[ -1 + (1-z_2)+(1-z_1)z_2+z_1z_2]=0
\]
The inner product of rows 1 and 4 is
\[- \sqrt{z_2(1-z_2)z_3z_4}+ (1-z_1)\sqrt{(1-z_2)z_2z_3z_4}+
z_1\sqrt{(1-z_2)z_2z_3z_4}=\]
\[ \sqrt{z_2(1-z_2)z_3z_4}[-1+(1-z_1)+z_1]=0.
\]
The inner product of rows 1 and 5 is
\[ -\sqrt{z_1(1-z_1)z_2z_3z_4}+\sqrt{z_1(1-z_1)z_2z_3z_4}]=0.\]
The inner product of rows 2 and 3 is
\[ -\sqrt{(1-z_4)(1-z_3)z_3}+(1-z_2)\sqrt{(1-z_4)(1-z_3)z_3}+\]
\[(1-z_1)z_2\sqrt{(1-z_4)(1-z_3)z_3}
+ z_1z_2\sqrt{(1-z_4)(1-z_3)z_3}=\]
\[ \sqrt{(1-z_4)(1-z_3)z_3)}[-1 +(1-z_2)+(1-z_1)z_2+ z_1z_2]=0.
\]

The inner product of rows 2 and 4 is
\[ -\sqrt{(1-z_4)z_2(1-z_2)z_3}+(1-z_1)\sqrt{(1-z_4)(1-z_2)z_2z_3}+
z_1\sqrt{(1-z_4)(1-z_2)z_2z_3}=0.
\]
The inner product of rows 2 and 5 is
\[ -\sqrt{(1-z_4)z_1(1-z_1)z_2z_3}+\sqrt{(1-z_4)z_1(1-z_1)z_2z_3}=0.\]

\section{The General Case}

For the general case, we take the $n\times n$ case and delete the first row and
add a first column of zeroes.  Then we add two new rows at the top of this matrix by
choosing a new variable $z_n$ and making row 1 as:
\[ \sqrt{1-z_n}\ \ \sqrt{(1-z_{n-1})z_n}\ \ \sqrt{(1-z_{n-2})z_{n-1}z_n}\ \ 
\sqrt{(1-z_{n-3})z_{n-1}z_{n-2}}\ \  \cdots \ \]
\[\sqrt{(1-z_1)z_2z_3\cdots z_{n}} \ldots 
 \sqrt{z_1z_2\cdots z_n}\]
and making row 2 as
\[ -\sqrt{z_n}\ \ \sqrt{(1-z_n)(1-z_{n-1})}\ \ \sqrt{(1-z_n)(1-z_{n-2})z_{n-1}}\ \ 
\sqrt{(1-z_n)(1-z_{n-3})z_{n-1}z_{n-2}}\ \ 
\cdots\]
\[ \ \ \sqrt{(1-z_n)(1-z_1)z_2z_3\ldots z_{n-1}}
\ \ \sqrt{(1-z_n)z_1z_2z_3\ldots z_{n-1}}.\]

\begin{remark}
It takes a significant amount of effort to show that we can produce all the
required Hessenberg unitary matrices and it does not seem to be important enough
to justify the effort since our intention is not to publish this
but just post it on the arXiv so it is available to researchers.
  So we will not address this here.  Basically, it can be done by induction on $n$
  and a case analysis of the placement of zeroes.
\end{remark}

\end{document}